\documentclass[12pt]{amsart}

\usepackage{hyperref}

\usepackage[letterpaper,hmargin=1in,vmargin=1in]{geometry}
\usepackage{paralist}

\usepackage[mathscr]{eucal}
\usepackage{indentfirst}

\RequirePackage{xfrac}
\RequirePackage{amscd}
\RequirePackage{dsfont}
\RequirePackage{amsthm}
\RequirePackage{amsmath}
\RequirePackage{amssymb}
\RequirePackage{mdwlist}
\RequirePackage{environ}
\RequirePackage{tikz-cd}
\RequirePackage{amsfonts}
\RequirePackage{mathrsfs}
\RequirePackage{fullpage}
\RequirePackage{graphicx}
\RequirePackage{fancyhdr}
\RequirePackage{lastpage}
\RequirePackage{enumerate}
\RequirePackage{asymptote}
\RequirePackage{extramarks}

\newcommand{\Z}{\mathds{Z}}

\newcommand{\kk}{\mathds{k}}
\newcommand{\p}{\mathfrak{p}}
\newcommand{\q}{\mathfrak{q}}
\newcommand{\m}{\mathfrak{m}}
\newcommand{\ann}{\mathrm{Ann}}

\newcommand{\spec}{\mathrm{Spec}}

\newcommand{\A}{\mathds{A}}

\renewcommand{\P}{\mathds{P}}
\newcommand{\xto} {\xrightarrow}

\newcommand{\im}{\mathrm{Im}}

\newcommand{\Hom}{\mathrm{Hom}}
\newcommand{\coker}{\mathrm{coker}}
\renewcommand{\ker}{\mathrm{ker}}

\newtheorem{theorem}{Theorem}[section]
\newtheorem{lemma}[theorem]{Lemma}
\newtheorem{proposition}[theorem]{Proposition}
\newtheorem{corollary}[theorem]{Corollary}

\theoremstyle{definition}
\newtheorem{definition}[theorem]{Definition}
\newtheorem{example}[theorem]{Example}
\newtheorem{procedure}[theorem]{Procedure}

\theoremstyle{remark}
\newtheorem{remark}[theorem]{Remark}
\newtheorem{caution}[theorem]{\bf Caution}

\renewenvironment{proof}[1][Proof]{\begin{trivlist}
\item[\hskip \labelsep {\bfseries #1}]}{\end{trivlist}}

\numberwithin{equation}{section}

\NewEnviron{cd}[1][small]{%
 \begin{center}\begin{tikzcd}[ampersand replacement=\&, column sep=#1, row sep=tiny]
  \BODY
 \end{tikzcd}\end{center}\ignorespacesafterend}

\NewEnviron{eq}{%
 \begin{equation*}\begin{split}
  \BODY
 \end{split}\end{equation*}\ignorespacesafterend}

\numberwithin{equation}{section}

\begin{document}

\title{Computations over Local Rings in Macaulay2}
\author{Mahrud Sayrafi}
\dedicatory{Thesis Advisor: David Eisenbud}
\address{University of California, Berkeley}
\email{mahrud@berkeley.edu}

\date{May 25, 2017}

\newcommand{\MM}{{\tt Macaulay2} }
\newcommand{\syz}{\text{Syz}}
\newcommand{\betti}{\beta}
\newcommand{\length}{\textrm{Length}}
\renewcommand{\b}{\bullet}
\renewcommand{\d}{\partial}

\begin{abstract}
 Local rings are ubiquitous in algebraic geometry. Not only are they naturally meaningful in a geometric sense, 
 but also they are extremely useful as many problems can be attacked by first reducing to the local case and taking 
 advantage of their nice properties. Any localization of a ring $R$, for instance, is flat over $R$. Similarly,
 when studying finitely generated modules over local rings, projectivity, flatness, and freeness are all equivalent.

 We introduce the packages {\tt PruneComplex}, {\tt Localization} and {\tt LocalRings} for {\tt Macaulay2}. 
 The first package consists of methods for pruning chain complexes over polynomial rings and their localization at 
 prime ideals. The second package contains the implementation of such local rings. Lastly, the third package 
 implements various computations for local rings, including syzygies, minimal free resolutions, length, minimal 
 generators and presentation, and the Hilbert--Samuel function.

 The main tools and procedures in this paper involve homological methods. In particular, many results depend on 
 computing the minimal free resolution of modules over local rings.
\end{abstract}

\maketitle

\tableofcontents


\section{Introduction}
 Local rings were first defined by Wolfgang Krull as a Noetherian ring with only one maximal ideal. The name,
 originating from German "Stellenring," points to the fact that such rings often carry the geometric information of 
 a variety in the neighborhood of a point, hence local.

 The procedures presented in this paper are described as pseudocodes and implemented in \MM, a computer algebra 
 software specializing in algebraic geometry and commutative algebra. In many situations, the power of \MM lies in 
 its algorithms for computing Gr\"obner bases using monomial orderings. In general, however, there are no known 
 monomial orders for local rings. The only exception to this is for localization with respect to a maximal ideal, 
 for which Mora's tangent cone algorithm provides a minimal basis suitable for computations. The main motivation for
 this work is studying local properties of algebraic varieties near irreducible components of higher dimension, such
 as the intersection multiplicity of higher dimensional varieties.


\subsection{Definitions}
 We begin with basic definitions. A ring shall always mean a commutative ring with a unit.

\begin{definition} 
  A ring $R$ is {\bf Noetherian} if any non-empty set of ideals of $R$ has maximal members.

  An $R$-module $M$ is {\bf Noetherian} if any non-empty set of submodules of $M$ has maximal members.
  \end{definition}

\begin{definition} 
  A ring $R$ is called a {\bf quasi-local ring} if it has only one maximal ideal.

  A Noetherian quasi-local ring $R$ is called a {\bf local ring}.
  \end{definition}

 There are many examples of local rings, as defined above, in the nature.
 For example, the ring $\kk[[x_1,x_2,\dots,x_n]]$ of formal power series with $n$ indeterminates is a (complete) 
 local ring where the maximal ideal is the ideal of all ring elements without a constant term.

\begin{definition} Let $R$ be a ring, $M$ an $R$-module, and $S\subset R$ a multiplicative closed set.

  The {\bf localization of $R$ at $S$} is the ring $R[S^{-1}] = S^{-1}R := \{(r,s) : r\in R, s\in S\}/\sim$, where 
  $(r,s)\sim(r',s')$ if and only if the is $t\in S$ such that $t(s'r-sr')=0$. The equivalence class of $(r,s)$ is 
  denoted $\tfrac rs$.
  Moreover, there is a canonical homomorphism $\varphi:R\to R[S^{-1}]$ given by $r\mapsto\tfrac r1$.
  This homomorphism induces a bijection between the set of prime ideals of $R[S^{-1}]$ and the set of prime ideals 
  of $R$ that do not intersect $S$%
  \footnote{In particular, the inclusion $\spec R[S^{-1}]\to\spec R$ is a flat morphism.}.

  The {\bf localization of $M$ at $S$} is the $R[S^{-1}]$-module $S^{-1}M$ defined similarly.
  In particular, we have $S^{-1}M= R[S^{-1}] \otimes_R M$ (see \cite[Lemma 2.4]{de}).
  \end{definition}

 In general, one can obtain quasi-local rings from any ring $R$ through {\bf localization at a prime ideal}:

\begin{definition} Let $R$ be a ring and $\p\subset R$ be any prime ideal.

  The {\bf localization of $R$ at $\p$} is defined as $R_\p=R[S^{-1}]$ where $S=R\setminus \p$.
  For any $R$-module $M$, denote $M_\p:=S^{-1}M=R_\p\otimes_R M$.
  \end{definition}

\begin{proposition}
  Let $R$ and $\p$ be as defined above, then $(R_\p, \p R_\p)$ is a quasi-local ring.
  \end{proposition}
\begin{proof}
  Recall that $\varphi:R\to R_\p$ induces a bijection between the set of prime ideals of $R_\p$ and the set of prime
  ideals of $R$ that do not intersect $S=R\setminus\p$, i.e., prime ideals of $R$ that are contained in $\p$.
  In particular, if $\q\subset R_\p$ is a prime, then $\varphi^{-1}(\q)\subset\p$, hence $\q\subset\p R_\p=\m$.
  Therefore $\m$ contains every prime ideal of $R_\p$ and is the unique maximal ideal.
  \qed
  \end{proof}

 This procedure is canonical in the sense that it satisfies the following universal property:

\begin{proposition}[Universal Property of Localization]
  Let $\varphi:R\to R[S^{-1}]$ be given by $r\mapsto r/1$.
  Suppose there is a ring $R'$ along with a homomorphism $\psi:R\to R'$ such that for any $s\in S$, $\psi(s)$ is a 
  unit in $R'$, then there is a unique homomorphism $\sigma:R[S^{-1}]\to R'$ such that $\psi=\sigma\varphi$
  (see \cite[pp. 60]{de}). 
  \end{proposition}

 Another useful property of localization is flatness:

\begin{definition}
  An $R$-module $F$ is {\bf flat} if $-\otimes_R F$ is an exact functor. 
  That is, if
  $$0\to M'' \to M \to M' \to 0$$
  is an exact sequence of $R$-modules, then 
  $$0\to M''\otimes_R F \to M\otimes_R F \to M'\otimes_R F \to 0$$
  is again an exact sequence.
  \end{definition}

\begin{lemma}[{\cite[Proposition 2.5]{de}}]\label{flat}
  Let $R$ be a ring and $S\subset R$ be a multiplicative closed set.
  The localization $R[S^{-1}]$ is a flat $R$-module.
  \end{lemma}

 Later, we will see conditions under which this functor is faithfully flat. For now, this lemma has various
 interesting applications:

\begin{corollary} Let $R$ and $S$ be as above and let $\p$ be a prime. 
  \begin{enumerate}
    \item For any ideal $I\subset R$, $(R/I)_\p = R_\p/I_\p$.
    \item Recall that for any $R$-modules $M$ and $N$, $\Hom_R(M,N)$ is the abelian group of $R$-module 
    homomorphisms from $M$ to $N$. In particular, $\Hom_R(M,N)$ is itself an $R$-module, and if $M$ is finitely 
    presented, we have an isomorphism $\Hom_{R[S^{-1}]} (S^{-1}M, S^{-1}N) \cong S^{-1}\Hom_R (M, N)$.
    \item For ideals $I,J\subset R$, recall that $J:I = \{r\in R : rI\subset J\}$ is the quotient ideal. 
    More generally, we can write $J:I\cong\Hom_R(R/I,R/J)$. 
    In particular, $J_\p:I_\p = (J:I)_\p$. 
    \end{enumerate}
  \end{corollary}
\begin{proof}
  We present the proofs as they exemplify the usefulness of flatness.
  \begin{enumerate}
    \item Consider the exact sequence of $R$-modules $0\to I\to R\to R/I \to 0$. Since $-\otimes_R R_\p$ is flat, we
    have another exact sequence $0\to I_\p \to R_\p \to (R/I)_\p \to 0$. Exactness of this sequence gives 
    $(R/I)_\p = \coker(I_\p \to R_\p)=R_\p/I_\p$.
    \item See \cite[Proposition 2.10]{de}.
    \item Using (1) and (2) we get:
     $$J_\p:I_\p = \Hom_{R_\p}(R_\p/I_\p,R_\p/J_\p) 
                   = \Hom_{R_\p}\big((R/I)_\p, (R/J)_\p\big) 
               \cong \Hom_R(R/I, R/J)_\p
                   = (J:I)_\p.$$
    \end{enumerate}
  Similar techniques will be used in future sections to prove correctness of procedures.
  \qed
  \end{proof}


\subsubsection{Geometric Intuition}
 At this point, it is appropriate to elaborate on the geometric picture behind localization:
 Let $A=\kk[x_1,\dots,x_r]$ be the polynomial ring with $r$ indeterminates and consider an affine variety $X$ in 
 $\A_\kk^r=\spec A$, the affine $r$-space.
 Suppose $R$ is the affine coordinate ring of $X$ and 
 take any point $P\in X$ corresponding to a prime ideal $\p\subset R$ of functions vanishing at $P$.
 Now, in order to study the behavior of $X$ ``near'' $P$, we can invert all functions in $R$ that do not vanish at 
 $P$, i.e., adjoin inverses for all functions in $R\setminus\p$. This process yields exactly the local ring $R_\p$, 
 the coordinate ring of the neighborhood of $P$ in $X$.

 Clearly the local ring of $R$ at a prime $\p$ is, in general, not finitely generated over $R$. This is the main 
 obstacle in developing computational methods (such as Gr\"obner bases) for local rings.
 However, certain computations remain tractable when we consider finitely generated modules over Noetherian rings.


\subsection{Preliminaries}
 In the next few definitions we describe the main tools used in our computations.

\begin{definition} Let $R$ be a ring and $M$ an $R$-module.

  A {\bf projective resolution of $M$} is an acyclic complex of $R$-modules
  $$P_\bullet:\cdots\to P_2\to P_1\to P_0\to M\to 0$$
  such that each $P_i$ is a projective $R$-module. In particular, every $R$-module has a projective resolution.
  \end{definition}

 Resolutions are used widely in homological algebra to construct invariants of objects.
 In particular, using the following proposition we know that every finitely generated module over a local ring has 
 a free resolution. 

\begin{proposition}
  \label{prop:fpf}
  For finitely generated modules over local rings, flatness, projectivity and freeness are all equivalent.
  \end{proposition}
\begin{proof}
  See \cite[Theorem 7.10]{MH}. 
  \qed
  \end{proof}

 An extraordinarily useful tool in the theory of local rings is the following lemma:

\begin{lemma}[{T. Nakayama, cf. \cite[Corollary 4.8]{de}}]\label{nakayama}
  Let $(R,\m)$ be a quasi-local ring and let $M$ be a finite $R$-module. 
  Then a subset $u_1,\dots,u_n$ of $M$ is a generating set for $M$ if and only if the set of residue classes 
  $\{u'_i\}$ is a generating set for $M/\m M$ over the field $R/\m$%
  \footnote{Nagata, following Nakayama, credits this lemma to W. Krull and G. Azumaya, whose generalization of this 
  lemma holds for any ring $R$ and its Jacobson radical $\m$, defined as the intersection of all maximal ideals of 
  $R$. Note that in local rings the Jacobson radical is the same as the unique maximal ideal}.
  \end{lemma}

\begin{corollary}\label{cor:nak}
  Let $(R,\m)$ and $M$ be as above. 
  \begin{enumerate}
    \item If $M/\m M = 0$, then $M=0$.
    \item If $N$ is a submodule of $M$ such that $M=\m M + N$, then $M=N$.
    \item Any generating set for $M$ contains a minimal generating set for $M$ as a subset; 
    if $u_1,\dots,u_m$ and $v_1,\dots,v_n$ are two minimal generating sets of $M$, then $m=n$ and there is an 
    invertible $n\times n$ matrix $T$ over $R$ such that $\bar v=\bar uT$ (see \cite[Corollary 5.3]{nagata}).
    \end{enumerate}
  \end{corollary}
\begin{proof}
  See \cite[Section 2]{MH}.
  \qed
  \end{proof}

 This corollary implies that {\it minimal free resolutions} are well defined over local rings, but first we need to
 make a new definition:

\begin{definition} Let $(R,\m)$ and $M$ be as above.

  Let $m_1,\dots,m_n$ be a generating set for $M$.
  The {\bf relation module of the elements $m_i$} is the set $N$ of elements $(r_1,\dots,r_n)\in R^n$ such that 
  $\sum r_i m_i=0$. Clearly $N$ has an $R$-module structure. In particular, there is an exact sequence:
  $$0\to N\to R^n \to M \to 0.$$

  When $m_i$ form a minimal generating set for $M$, $N$ is called the {\bf relations module of $M$}.

  The {\bf $i$-th Syzygy module of $M$}, denoted as $\syz_R^i(M)$, is defined as the relation module of the 
  $(i-1)$-th Syzygy module of $M$ when $i>0$ and $M$ when $i=0$.
  \end{definition}

 Note that the definition above can be extended to other rings, but the resulting module depends on the choice of
 the generating set. In particular, the corollary above implies that over local rings the syzygies of finite modules
 are unique up to isomorphism:

\begin{theorem}[{\cite[Theorem 26.1]{nagata}}]\label{thm:syz}
  Let $(R, \m)$ be a local ring and $M$ be a finite $R$-module.
  The $i$-th Syzygy module of $M$ is unique up to isomorphism. In particular, if $N$ is the relation module of an
  arbitrary generating set $u_i$ for $M$, then $N=\syz_R^1(M)\oplus R^k$ for some $k$.
  \end{theorem}
\begin{proof} Follows from \ref{cor:nak} (2). 
  \qed
  \end{proof}

\begin{remark}
  Suppose $\{m_i\}$ is a generating set for $M$ with $n_0$ elements and the relation module $\syz_R^1(\{m_i\})$ has 
  a generating set with $n_1$ elements. Then, using the exact sequence in the definition above we can find an exact
  sequence:
  $$R^{n_1}\to R^{n_0}\to M\to 0.$$
  Repeating this process inductively for $\syz_R^i(\{m_i\})$, we can construct a free resolution for $M$.
  \end{remark}

\begin{definition}\label{def:minres}
  Let $(R, \m)$ be a local ring and let $M$ be an $R$-module.

  A {\bf minimal free resolution $F$ of $M$} is a resolution:
  $$F_\bullet:\cdots\xto{\d}F_2\xto{\d}F_1\xto{\d}F_0\to M\to 0$$
  such that for every differential, $R/\m\otimes_R \d=0$. That is to say, there are no units in the differentials.
  \end{definition}

\begin{remark}
  Theorem \ref{thm:syz} implies that any resolution of $M$ contains the minimal free resolution as a summand and 
  the minimal free resolution is unique up to changes of basis.
  \end{remark}

 Minimal free resolutions, when they exist, capture many structural invariants of modules, which is why many
 results are limited to the local and graded cases where Nakayama's Lemma holds. In particular, everything that 
 follows also applies to the graded case by setting $\m$ to be the irrelevant ideal. That is, when 
 $R=\oplus_{i\geq 0}R_i$ is a finitely generated graded algebra over the field $k=R_0$ and $M$ is a finitely 
 generated graded $R$-module.

\newcommand{\rk}{\text{rk}}
\begin{definition}\label{betti}
  Let $(R, \m)$, $M$, $F_\bullet$ be as above and suppose $F_i\cong R^{\oplus r_i}$.

  The {\bf $i$-th Betti number of $M$}, $\betti^R_i(M)$ is defined as the rank of the $i$-th module in the minimal 
  free resolution of $M$:
  $$\betti^R_i(M) = \rk_R F_i = r_i.$$
  Equivalently, this is the minimal number of generators of the $i$-th Syzygy module of $M$:
  $$\syz_R^i(M) = \ker\d_{i-1} = \im\d_i \cong \coker\d_{i+1}.$$
  \end{definition}


\subsection{Artinian Local Rings}
 Besides dimension and projective dimension, a very useful tool in building invariants for local rings and modules
 over them is the length.
 Another useful construction in commutative ring theory is completion. Localization at a prime ideal followed by 
 completion at the maximal ideal a common step in many proofs. Many problems, can be attacked by first reducing to 
 the local case and then to the complete case. In particular, Artinian local rings are automatically complete and
 therefore share the nice properties of complete local rings. As an example, when $(R,\m)$ is local, $\hat R$ is 
 local and faithfully flat over $R$.

 In this section we review important results leading up to the definition of the 
 Hilbert-Samuel function.


\begin{definition} 
  A ring $R$ is {\bf Artinian} if any descending chain of ideals is finite.
  \end{definition}

\begin{definition} 
  A {\bf composition series} of an $R$-module $M$ is a chain of inclusions:
  $$M=M_0\supset M_1\supset\cdots\supset M_n,$$
  where the inclusions are strict and each $M_i/M_{i+1}$ is a simple module.

  The {\bf length of $M$} is defined as the least length of a composition series for $M$, or $\infty$ if $M$ has no
  finite composition series.
  \end{definition}

\begin{theorem}[{\cite[Theorem 2.14]{de}}]
  Let $R$ be a ring $R$. The following conditions are equivalent:
  \begin{itemize}
    \item $R$ is Artinian.
    \item $R$ has finite length as an $R$-module. 
    \item $R$ is Noetherian and all the prime ideals in $R$ are maximal.
    \end{itemize}
  \end{theorem}

 Let $(R, \m)$ be an Artinian local ring. This theorem implies that $\m$ is its only prime ideal.

\begin{theorem}[{\cite[Theorem 2.13]{de}}]
  Let $R$ be any ring and $M$ be as above. $M$ has a finite composition series if and only if $M$ is Artinian.
  Moreover, every composition series for $M$ has the same length.
  \end{theorem}

 In particular, over local rings we have $M_i/M_{i+1}\cong R/\m$ for all $i$. Note that over an Artinian local ring,
 the chain $\m^iM\supset\m^{i+1}M\supset\cdots$ eventually terminates. That is to say, $\m^n M=\m^{n+1}M$, hence by 
 Lemma \ref{nakayama}, $\m^n M/\m^{n+1} M=0$ implies that $\m^n M=0$. Therefore, every composition series is 
 equivalent to a refinement of the chain:
 $$M\supset \m M \supset\cdots\supset \m^n M=0.$$
 Therefore, we can compute length by computing the sum of lengths of each term in the sequence above. Recall that by
 Lemma \ref{nakayama}, $\m^i M$ and $R/\m$-vector space $\m^i M/\m^{i+1}M$ have the same length and number of
 generators. Moreover, the length and number of basis elements of a vector space are equal, so the length of $\m^iM$
 is the size of its minimal generating set.

\begin{corollary}[{\cite[Corollary 2.17]{de}}]
  Let $R$ be a Noetherian ring and $M$ be a finite $R$-module. The following conditions are equivalent:
  \begin{itemize}
    \item $M$ has finite length.
    \item All the primes that contain the annihilator of $M$ are maximal.
    \item $R/\ann(M)$ is Artinian.
    \end{itemize}
  \end{corollary}

 Lastly:

\begin{corollary}[{\cite[Corollary 2.18]{de}}]\label{cor:mengyuan}
  Let $R$ and $M$ be as above. Suppose $I=\ann(M)$ and $\p$ is a prime of $R$ containing $I$.
  Then $M_\p=M\otimes_R R_\p$ has finite length if and only if $\p$ is minimal among primes containing $I$.
  \end{corollary}

 In particular, when $M=R/I$, we conclude that $M_\p=R_\p/I_\p$ is Artinian if and only if $\p$ is the minimal prime 
 over $I$.

\begin{definition} Let $(R, \m)$ be a local ring and let $M$ be a finitely generated $R$-module.

  Ideal $\q\subset\m$ is a {\bf parameter ideal for $M$} if $M/\q M$ has finite length.
  Equivalently, $\q$ is a parameter ideal for $M$ if: $\text{rad}(\text{ann}(M/\q M))=\m.$

  A {\bf system of parameters} is a sequence $(x_1,\dots,x_d)\subset\m$ such that $\m^n\subset(x_1,\dots,x_d)$
  for $n>>0$.
  \end{definition}

 Geometrically, for local ring of a variety $X$ at $P$, a system of parameters is a local coordinate system for $X$
 around $P$.


\begin{definition}\label{def:HS}
  Let $R$ and $M$ be as above and let $\q$ be a parameter ideal for $M$.

  The {\bf Hilbert-Samuel function of $M$ with respect to parameter ideal $\q$} is:
  $$H_{\q,M}(n):=\length(\q^nM / \q^{n+1}M).$$
  Note that $\q^nM / \q^{n+1}M$ is a module over the Artinian ring $R/(\q+\text{ann}(M))$, hence the length is 
  finite.
  \end{definition}

\begin{corollary} Let $R$ and $M$ be as above, then $\m$ is a parameter ideal for $M$. In particular:
  $$\length_R M = \sum_{i=0}^n \length_R(\m^iM/\m^{i+1}M) = \sum_{i=0}^n H_{\m,M}(i).$$
  \end{corollary}






\section{Elementary Computations}

 The procedures in this section are implemented in \MM language and are available across three packages: 
 {\tt Localization}, {\tt LocalRings}, and {\tt PruneComplex} \cite{github}. In order to run the examples, it 
 suffices to load {\tt LocalRings} as follows:

\begin{verbatim}
Macaulay2, version 1.9.2
i1 : needsPackage "LocalRings"
  \end{verbatim}

 The first step in performing computations over local rings is defining a proper data structure that permits
 defining the usual objects of study, such as ideals, modules, and complexes over local rings. We have implemented 
 localization of polynomial rings with respect to prime ideals as a limited type of the field of fractions in 
 {\tt Localization}:

\begin{verbatim}
i2 : R = ZZ/32003[x,y,z];
i3 : P = ideal"x,y,z";
i4 : RP = localRing(R, P)
o4 : LocalRing
  \end{verbatim}

 Let $M_\p$ be a module over the local ring $R_\p$. Our first computation is to find a free resolution of $M_\p$.

\begin{proposition} Let $R_\p$ and $M_\p$ be as above.
  Suppose we have a module $N$ over the parent ring $R$ such that $N\otimes_R R_\p=M_\p$ and suppose $C_\b$ is a
  free resolution of $N$. Then $C_\b\otimes_R R_\p$ is a free resolution for $M_\p$.
  \end{proposition}
\begin{proof}
  Recall from Lemma \ref{flat} that $R_\p$ is a flat module over $R$, hence the the sequence $C_\b\otimes_R R_\p$
  remains exact and is a resolution for $N\otimes_R R_\p=M_\p$. Moreover, if $F$ is a free $R$-module then
  $F\otimes_R R_\p$ is a free $R_\p$-module. Therefore $C_\b\otimes_R R_\p$ is a free resolution for $M_\p$
  \qed
  \end{proof}

 Fortunately, efficient procedures for finding free resolutions for homogeneous and non-homogeneous modules over 
 polynomial rings using Gr\"obner basis methods already exist. Therefore, one way to find a free resolution of 
 $M_\p$ is to begin with finding a suitable module $N$, then find a free resolution $C_\b$ for $N$ and tensor it 
 with $R_\p$ to get a free resolution for $M_\p$%
 \footnote{Note that this methods would also apply to localization with respect to a multiplicative closed set.}.
 Therefore, we need to show that a suitable module $N$ can be efficiently constructed.

\begin{remark} 
  Finitely generated modules can be defined in \MM in four different ways:
  \begin{itemize}
    \item Free module: $F=R^n$.
    \item Submodule of a free module; such modules arise as the image of a map of free modules:
      $$R^m \xto{\begin{bmatrix}g_1 & g_2 & \cdots & g_m \end{bmatrix}} G\to 0.$$
      The matrix of generators $\begin{bmatrix}g_1 & g_2 & \cdots & g_m \end{bmatrix}$ captures the data in $G$.
    \item Quotient of a free module; such modules arise as the cokernel of a map of free modules:
      $$R^m \xto{\begin{bmatrix}h_1 & h_2 & \cdots & h_m \end{bmatrix}} R^n \to R^n/H\to 0.$$
      A matrix of relations $\begin{bmatrix}h_1 & h_2 & \cdots & h_m \end{bmatrix}$ captures the data in $R^n/H$.
    \item Subquotient modules; given two maps of free modules $g:R^m\to R^{m+n}$ with $G=\im(g)$ and 
      $h:R^n\to R^{m+n}$ with $H=\im(h)$, the subquotient module with generators of $g$ and relations of $h$ is 
      given by $(H+G)/H$. That is to say:
      $$R^m \xto{\begin{bmatrix} h_1 & h_2 &\cdots& h_m                          \end{bmatrix}}R^{m+n} 
            \xto{\begin{bmatrix} h_1 & h_2 &\cdots& h_m & g_1 & g_2 &\cdots& g_n \end{bmatrix}}\frac{H+G}{H}\to 0.$$
      Two matrices, one for relations in $H$ and one for generators of $G$ are needed to store $(G+H)/H$.
    \end{itemize}
  Note that the class of subquotient modules contains free modules and is closed under the operations of taking 
  submodules and quotients. In particular, submodules are subquotient modules with relations module $H=0$ and 
  quotients modules are subquotient modules with generator module $G=R^n$.
  \end{remark}

\begin{proposition}\label{lift}
  Using the notation above, consider a $R_\p$-module $M_\p$
  with generator module $G_\p$ and relation module $H_\p$, i.e. $M_\p=(G_\p+H_\p)/H_\p$,
  such that $\tfrac{g_1}{u_1},\dots,\tfrac{g_n}{u_n}$ generate $G_\p\subset R^r_\p$
  and       $\tfrac{h_1}{v_1},\dots,\tfrac{h_m}{v_m}$ generate $H_\p\subset R^r_\p$ 
  where     $h_i,g_i\in R_\p^r$ and $u_i,v_i\in R_\p^*=R\setminus\p$.
  Let $G$ and $H$ be the $R$-modules generated by the $g_i$ and $h_i$, respectively, and let $N=(G+H)/H$ be the
  $R$-module with generator module $G$ and relation module $H$. Then $N \otimes_R R_\p = M_\p$.
  \end{proposition}
\begin{proof}
  Consider the exact sequence of $R$-modules:
  $$0\to H\to H+G\to \frac{H+G}{H}\to 0.$$
  By Lemma \ref{flat} $R_\p$ is a flat $R$-module, so we have an exact sequence of $R_\p$-modules:
  $$0\to H\otimes_R R_\p\to (H+G)\otimes_R R_\p\to \Big(\frac{H+G}{H}\Big)\otimes_R R_\p\to 0.$$
  Hence:
  $$N\otimes_R R_\p = \Big(\frac{H+G}{H}\Big)\otimes_R R_\p
                    = \frac{(H+G)\otimes_R R_\p}{H\otimes_R R_p} 
                    = \frac{H\otimes_R R_\p + G\otimes_R R_\p}{H\otimes_R R_p} $$
  Note that both $G_\p$ and $H_\p$ are submodules of free module $R_\p^r$, therefore it suffices to show 
  $G\otimes_R R_\p=G_\p$ for any submodule of a free module. Observe that without loss of generality $G_\p$ can be 
  generated by $g_i$, as $u_i\in R_\p^*$.
  Note that $\m G_\p + G \otimes_R R_\p = \sum \m g_i + \sum R_\p g_i = \sum R_\p g_i = G_\p$.
  Therefore $G_\p = G \otimes_R R_\p$ and $H_\p = H \otimes_R R_\p$ by Corollary \ref{cor:nak}.
  \qed
  \end{proof}

 
 Finding the module $M$ is done using the ${\tt liftUp}$ procedure in {\tt LocalRings}.

\begin{procedure}[{\tt liftUp}]
  Given a matrix $M_\p$ over local ring $R_\p$, returns a matrix $M$ over $R$ such that $M_\p=R_\p\otimes_R M$.
  This is done by multiplying each column of $M_\p$ with the least common denominator of elements in that column,
  then formally lifting up the matrix to $R$.
\begin{verbatim}
Input:  matrix MP over RP
Output: matrix M  over R
Begin
    for i from 1 to number of columns of MP do
      col <- i-th column of MP
      d   <- LCM of denominators of elements of col
      MP  <- MP with the i-th column multiplied by d
    M <- formally lift MP from RP to R
    RETURN M
End
  \end{verbatim}
  As shown above, this is enough to lift any module. Given a matrix $M_\p$ over local ring $R_\p$, we can find maps
  $g:R_\p^m\to R_\p^r$ and $h:R_\p^n\to R_\p^r$ such that $G_\p=\im(g)$ is the relation module of $M_\p$ and
  $H_\p=\im(h)$ is the generator module of $M_\p$. The matrix representing the maps is found using \MM commands 
  {\tt generators} and {\tt relations}. Then the matrices $g$ and $h$ can be lift up, which then give $G$ and $H$ as
  their images. After that we have the lift up as $M=(G+H)/H$. In \MM, the command {\tt subquotient} returns the
  subquotient module given the matrices $g$ and $h$.
\begin{verbatim}
Input:  module MP over RP
Output: module M  over R
Begin
    g <- generators MP
    g <- liftUp g
    h <- relations MP
    h <- liftUp h
    M <- subquotient(g, h)
    RETURN M
End
  \end{verbatim}
  \end{procedure}

\begin{example} Computing free resolution of a Gorenstein ideal of projective dimension 2 with 5 generators
  \cite[Example 5, pp. 127]{EB}:
\begin{verbatim}
i5 : use RP;
i6 : IP = ideal(x^3 + y^3, x^3 + z^3, x*y/(z+1), x*z/(y+1), y*z/(x+1))
             3    3   3    3   x*y    x*z    y*z
o6 = ideal (x  + y , x  + z , -----, -----, -----)
                              z + 1  y + 1  x + 1
o6 : Ideal of RP
i7 : I = liftUp IP
             3    3   3    3
o7 = ideal (x  + y , x  + z , x*y, x*z, y*z)
o7 : Ideal of R
i8 : C = res I
      1      5      5      1
o8 = R  <-- R  <-- R  <-- R  <-- 0
     0      1      2      3      4
i9 : C ** RP
       1       5       5       1
o9 = RP  <-- RP  <-- RP  <-- RP  <-- 0
     0       1       2       3       4
  \end{verbatim}
  \end{example}

\begin{caution}
  A priori there is no reason to believe that this resolution is minimal.
  \end{caution}

 Recall from Theorem \ref{thm:syz} that a free resolution $F_\b$ is minimal when there are no units in any of the 
 differentials $\d:F_i\to F_{i-1}$. Therefore we can minimize resolutions by iteratively removing all units from the
 differentials while preserving the complex.

\begin{proposition} 
  Using the notation as before, suppose we have a free resolution for $M$:
  \begin{cd}
    \cdots\ar[r] \& F_{i+1} \ar[r,"\d_{i+1}"] \& F_i \ar[r, "\d_i"] \& F_{i-1} \ar[r,"\d_{i-1}"] \& F_{i-2} \ar[r] \&
    \cdots\ar[r] \& F_0     \ar[r]            \& M   \ar[r]         \& 0.
    \end{cd}
  If the matrix of $\d_i$ contains a unit, we can construct a free resolution $\tilde F_\b$ for $M$ such that 
  $F_i\cong\tilde F_i\oplus R$ and $F_{i-1}\cong\tilde F_{i-1}\oplus R$.
  \end{proposition}
\begin{proof}
  Suppose $F_i\cong R^{\oplus r_i}$ and that for each $i$ we can represent $\d_i$ as a $r_{i-1}\times r_{i}$ matrix 
  with entries in $R$. Then $M(\d_i)$ can be written as:
  \[ M(\d_i) =
   \begin{bmatrix}
     a_{1,1}       & \dots   & a_{1,r_i}       \\
     \vdots        & a_{m,n} & \vdots          \\
     a_{r_{i-1},1} & \dots   & a_{r_{i-1},r_i} \\
     \end{bmatrix}.
  \]
  such that the element $u=a_{m,n}$ is a unit. Note that if $S$ and $T$ are change of coordinate matrices for $F_i$
  and $F_{i-1}$ respectively, then without loss of generality we can consider the free resolution:
  \begin{cd}
    \cdots\ar[r] \& F_{i+1} \ar[rr,"\d_{i+1}S^{-1}"] \&\& F_i \ar[rr, "S\d_iT"] \&\& F_{i-1} \ar[rr,"T^{-1}\d_{i-1}"] 
    \&\& F_{i-2} \ar[r] \&
    \cdots\ar[r] \& F_0     \ar[r]                  \& M   \ar[r]           \& 0.
    \end{cd}
  In particular, we can find invertible matrices $S$ and $T$ such that $S\d_iT$ has $u$ in the top left corner of 
  the matrix. This can be done by swapping the $n$-th and first columns of $M(\d_i)$ while swapping $n$-th and first
  rows of $M(\d_{i+1})$, and similarly swapping the $m$-th and first rows of $M(\d_i)$ while swapping $m$-th and 
  first columns of $M(\d_{i-1})$.

  Now that $u=a_{1,1}$, let $c_j$ denote the $j$-th column. For each $1<j\leq r_i$, we can use changes of coordinate
  to add $-a_{1, j}c_1/u$ to the column $c_j$ without losing exactness. As a result, the top row of the matrix is 
  now all zero except for $u$ in the top right corner. Similarly, denoting the $j$-th row by $v_j$, we can zero out 
  the first column (with the exception of $u$) by adding $-a_{j, 1}v_1/u$ to the $j$-th row for $1<j\leq r_{i-1}$.
  The result is:
  \[ M(\d'_i) =
   \begin{bmatrix}
     u		& 0		& 0		& \dots	 & 0			\\
     0		& a_{1,1}	& a_{1,2}	& \dots	 & a_{1,r_i}		\\
     0		& a_{2,1}	& a_{2,2}	& \dots	 & a_{2,r_i}		\\
     \vdots	& \vdots	& \vdots	& \ddots & \vdots		\\
     0		& a_{r_{i-1},1}	& a_{r_{i-1},2}	& \dots	 & a_{r_{i-1},r_i}	\\
     \end{bmatrix}.
  \]
  That is to say, $F_\b$ can be decomposed as direct sum of two exact sequences:
  \begin{cd}
    \cdots  \ar[r]	\&
    F_{i+1} \ar[rr,"\d'_{i+1}"]	\&\&	R^{\oplus r_i-1}	\ar[rr, "\d'_i"]	\&\&	R^{\oplus r_{i-1}-1}
    \ar[rr,"\d'_{i-1}"]		\&\&	F_{i-2}	\ar[r]	\& \cdots \ar[r] \& F_0 \ar[r]	\& M \ar[r] \& 0 \\
    \cdots  \ar[r]	\& 0       \ar[rr]	\&\& R \ar[rr, "\cdot u"]	\&\& R \ar[rr]	\&\& 0 \ar[r]	\& \cdots
    \end{cd}
  Therefore, without loss of generality, we can remove the first row of $M(\d_{i+1})$ and first column of $M(\d_i)$
  as well as first column of $M(\d_{i-1})$ and first row of $M(\d_i)$ to get the maps for $\tilde F_\b$.
  \qed
  \end{proof}

 Note that we can also perform the same operation on resolutions of homogeneous or inhomogeneous modules, but in the
 inhomogeneous case minimal resolutions are not well-defined. This procedure is implemented in {\tt PruneComplex}:

\begin{procedure}[{\tt pruneUnit}]
  Given a chain complex as a list of compatible matrices $C$ and index $i$ of one of the differentials such that 
  the element in the last row and column is a unit, this procedure prunes the resolution as described by the
  proposition above
\begin{verbatim}
Input:  chain complex C and integer i
Output: chain complex C
Begin
  M <- C_i
  m <- number of rows    of M
  n <- number of columns of M
  u <- M_(m,n)
  if C_(i-1) exists then
     C_(i-1) <- C_(i-1) with first column deleted
  for c from 2 to n do
    C_i <- C_i with first column times (M_(1,c)/u) subtracted from column c
  C_i <- C_i with first row    deleted
  C_i <- C_i with first column deleted
  if C_(i+1) exists then
     C_(i+1) <- C_(i+1) with first row    deleted
  RETURN C
End
  \end{verbatim}
  \end{procedure}

 One direction for improving this procedure is removing multiple units simultaneously by collecting them in a 
 square matrix concentrated in the lower right corner.
 In general, before using this procedure, we need to find a suitable unit and move it to the bottom right corner of 
 the matrix. By suitable unit we mean a unit that has the lowest number of terms in its row and column. This 
 distinction is important in order to keep the procedures efficient.


\begin{procedure}[{\tt pruneDiff}]
  Given a chain complex as a list of compatible matrices $C$ and index $i$ of one of the differentials, completely 
  prunes the differential by iteratively finding a suitable unit, moving it to the end, then calling {\tt pruneUnit}
  to remove the unit.
\begin{verbatim}
Input:  chain complex C, integer i
Output: chain complex C
Begin
    while there are units in C_i do
      (m,n) <- the coordinates of the unit with the sparsest row and column in C_i
      C_(i-1) <- C_(i-1) with column m and first column swapped
      C_i     <- C_i     with row    m and first row    swapped
      C_i     <- C_i     with column n and first column swapped
      C_(i+1) <- C_(i+1) with row    n and first row    swapped
      pruneUnit(C, i)
    RETURN C
End
  \end{verbatim}
  \end{procedure}

 Note that the resulting differential does not contain any units, however, it is not necessarily a differential of 
 the minimal complex, as there may be units in the adjacent differentials.


 Finally, repeating the previous procedure iteratively for each differential, we can remove all units from all 
 differentials to minimize the free resolution. This task is accomplished in the {\tt pruneComplex} procedure.
%
%
 Note that in general the order of pruning differentials is arbitrary, but in special cases there may be a right
 choice.




\subsection{Is the Smooth Rational Quartic a Cohen-Macaulay Curve?}\label{CM}
\begin{example}
 In this example we test whether the smooth rational quartic curve is locally a Cohen-Macaulay.
 Define the rational quartic curve in $\P^3$:
\begin{verbatim}
i2 : R = ZZ/32003[a..d];
i3 : I = monomialCurveIdeal(R,{1,3,4})
                        3      2     2    2    3    2
o3 = ideal (b*c - a*d, c  - b*d , a*c  - b d, b  - a c)
o3 : Ideal of R
  \end{verbatim}
  Compute a free resolution for $I$:
\begin{verbatim}
i4 : C = res I
      1      4      4      1
o4 = R  <-- R  <-- R  <-- R  <-- 0
     0      1      2      3      4
o4 : ChainComplex
  \end{verbatim}
  Localize the resolution at the origin:
\begin{verbatim}
i5 : M = ideal"a,b,c,d"; -- maximal ideal at the origin
o5 : Ideal of R
i6 : RM = localRing(R, M);
i7 : D = C ** RM;
i7 : E = pruneComplex D
       1       4       4       1
o8 = RM  <-- RM  <-- RM  <-- RM
     0       1       2       3
o8 : ChainComplex
  \end{verbatim}
  That is to say, the rational quartic curve is not Cohen-Macaulay at the origin $\m$.
  Therefore the curve is not Cohen-Macaulay in general. Now we localize with respect to a prime ideal:
\begin{verbatim}
i9 : P = ideal"a,b,c";   -- prime ideal
o9 : Ideal of R
i10 : RP = localRing(R, P);
i11 : D' = C ** RP;
i12 : E' = pruneComplex D'
        1       2       1
o12 = RP  <-- RP  <-- RP  <-- 0
      0       1       2       3
o12 : ChainComplex
  \end{verbatim}
  However, the curve is Cohen-Macaulay at the prime ideal $\p$ (and in fact any other prime ideal).
  \end{example}


\section{Other Computations}

 Using the procedures described in the previous section many other procedures already implemented for polynomial 
 rings can be extended to work over local rings. In particular, this works well when the result of the procedure can
 be described in free resolution. The common trick here is to use {\tt liftUp} to lift the object to a polynomial 
 ring, perform the procedure, tensor the resolution describing the result with $R_\p$, then prune the resolution.

\subsection{Computing Syzygy Modules}
 Perhaps the easiest example of this trick is in finding a minimal resolution for a module, as demonstrated in
 Section \ref{CM}. The steps involved in that example used to implement the following procedure in {\tt LocalRings}:

\begin{procedure}[{\tt resolution}]
 Given a module $M$ over local ring $R_\p$ returns a minimal free resolution for $M$. The first step is finding a
 map of free modules $f:R_\p^m\to R_\p^n$ such that $M\cong\coker(f)$. This is accomplished using the \MM command
 {\tt presentation}, which returns the matrix of the map $f$.
\begin{verbatim}
Input:  module M
Output: chain complex C
Begin
    RP <- ring M
    f  <- presentation M
    f' <- liftUp f
    M' <- coker f'
    C  <- resolution M'
    CP <- C ** RP
    CP <- pruneComplex CP
    RETURN CP
End
  \end{verbatim}
  \end{procedure}

 To further demonstrate the trick mentioned above, we computing the first syzygy module over a local ring in the 
 next example, then give a short procedure for computing the syzygy.

\begin{example}
  Define a the coordinate ring of $\A^6$:
\begin{verbatim}
i2 : R = ZZ/32003[vars(0..5)];
  \end{verbatim}
  Define the local ring at the origin:
\begin{verbatim}
i6 : M = ideal"a,b,c,d,e,f";
i7 : RM = localRing(R, M);
  \end{verbatim}
  Consider the cokernel module of an arbitrary matrix $f$ over $R_\m$:
  $$R_\m^3\xto{f} R_\m^3 \to N \to 0$$
\begin{verbatim}
i9 : f
o9 = | -abc+def  0       -b3+acd   |
     | 0         abc-def ab2-cd2-c |
     | ab2-cd2-c -b3+acd 0         |
              3        3
o9 : Matrix RM  <--- RM
  \end{verbatim}
  Lift the matrix to $R$ and compute its first two syzygies:
\begin{verbatim}
i10 : f' = liftUp f;
              3       3
o10 : Matrix R  <--- R
i11 :  g' = syz f';
              3       1
o11 : Matrix R  <--- R
i12 :  h' = syz g';
              1
o12 : Matrix R  <--- 0
  \end{verbatim}
  That is to say, the lift of our module has a free resolution:
  $$0\xto{h'} R^1 \xto{g'} R^3 \xto{f'} R^3 \to N' \to 0.$$
  Now we tensor the differentials with $R_\m$:
\begin{verbatim}
i13 :  g = g' ** RM;
               3        1
o13 : Matrix RM  <--- RM
i14 :  h = h' ** RM;
               1
o14 : Matrix RM  <--- 0
i15 :  C = {mutableMatrix g, mutableMatrix h};
  \end{verbatim}
  So our module has a free resolution:
  $$0\xto{h} R_\m^1 \xto{g} R_\m^3 \xto{f} R_\m^3 \to N \to 0.$$
  Now we prune the resolution:
\begin{verbatim}
i16 :  pruneDiff(C, 1)
o16 = {| b3-acd    |, 0}
       | ab2-cd2-c |
       | -abc+def  |
  \end{verbatim}
  That is to say, we have a resolution:
  $$0\xto{0} R_\m^1\xto{\begin{bmatrix}b^3-acd \\ ab^2-cd^2-c \\ def-abc\end{bmatrix}} R_\m^3 \to N\to 0.$$
  We can test that our result is a syzygy matrix:
\begin{verbatim}
i17 : GM = matrix C#0;
i18 : FM * GM == 0
o18 = true
  \end{verbatim}
  Recall from Definition \ref{betti} that $\syz_R^i(N)=\im\d_i$ is the syzygy module of $M$. So the first syzygy 
  module of $N$ is given by:
\begin{verbatim}
i26 : image GM
o26 = image | b3-acd    |
            | ab2-cd2-c |
            | -abc+def  |
                                3
o26 : RM-module, submodule of RM
  \end{verbatim}
  \end{example}

\begin{corollary} Let $R_\p$ and $M$ be as above. Suppose $N$ is the $R$-module described in \ref{lift}.
  Then $\syz_{R_\p}^1(M)=\syz_R^1(N)\otimes_R R_\p$.
  \end{corollary}
\begin{proof}
  Recall that the syzygy module of $M$ is the minimal module such that we have a short exact sequence
  $0\to \syz_R^1(N) \to R^n \to N \to 0$. By Lemma \ref{flat}, we have a short exact sequence
  $0\to \syz_R^1(N)\otimes_R R_\p \to R_\p^n \to N\otimes_R R_\p = M \to 0$. Which proves the corollary.
  \qed
  \end{proof}

 The steps above are implemented in the {\tt syz} procedure in {\tt LocalRings}:

\begin{procedure}[{\tt syz}]
 Given a matrix $M$ over local ring $R_\p$ returns the first syzygy matrix of $M$.
\begin{verbatim}
Input:  matrix M
Output: first syzygy matrix of M
Begin
    RP <- local ring of M
    f' <- liftUp M
    g' <- syz f'
    h' <- syz g'
    g <- g' ** RP
    h <- h' ** RP
    C <- {g, h}
    C <- pruneDiff(C, 1)
    RETURN C#0
End
  \end{verbatim}
  \end{procedure}


\subsection{Computing Minimal Generators and Minimal Presentation}


\begin{procedure}[{\tt mingens}]
 Given a module $M$ over local ring $R_\p$, returns the matrix of minimal generators of $M$. This is accomplished by
 lifting up the module, finding a resolution of length one (i.e., computing the syzygy once), then localizing the
 resolution and pruning it to ensure we have minimal generators.
\begin{verbatim}
Input:  module M
Output: matrix of minimal generators of M
Begin
    RP <- ring of M
-- Free Module:
    if M is a free module then
        RETURN generators M
-- Module defined by generators:
    if M is submodule of a free module then
        f <- generators M
        f' <- liftUp f
        g' <- syz f'
        g <- g' ** RP
        C <- {f, g}
        C <- pruneDiff(C, 1)
        RETURN C#0
-- Module defined by relations:
    if M is quotient of a free module then
        f <- relations M
        f' <- liftUp f
        g' <- syz f'
        g <- g' ** RP
        C <- {f, g}
        C <- pruneDiff(C, 1)
        RETURN C#0
-- Module defined by generators and relations:
    if M is a subquotient then
        f <- generators M
        g <- relations M
        f' <- liftUp f
        g' <- liftUp g
        h' <- modulo(f, g)
        h <- h' ** RP
        C <- {h}
        C <- pruneComplex C
        RETURN C#0
End
  \end{verbatim}
  The command {\tt modulo} returns a matrix $h$ whose image is the pre-image of the image of $g$ under $f$, i.e., 
  $\im(h)=f^{-1}(\im(g))$.
  \end{procedure}

 A very similar procedure can be used to compute minimal presentation of modules.

\begin{procedure}[{\tt minimalPresentation}]
  Given a module $M$ over local ring $R_\p$, returns the module $N\cong M$ with minimal generators and relations.
  This is accomplished by lifting up the module, finding a resolution of length two (i.e., computing the syzygy 
  twice), then localizing the resolution and pruning it to ensure we have minimal generators and relations.
\begin{verbatim}
Input:  module M
Output: module N with minimal generators and relations
Begin
    RP <- ring of M
-- Free Module:
    if M is a free module then
        N <- M
-- Module defined by generators:
    if M is submodule of a free module then
        f <- generators M
        f' <- liftUp f
        g' <- syz f'
        h' <- syz g'
        g <- g' ** RP
        h <- h' ** RP
        C <- {g, h}
        C <- pruneComplex C
        N <- coker C#0
-- Module defined by relations:
    if M is quotient of a free module then
        f <- relations M
        f' <- liftUp f
        g' <- syz f'
        g <- g' ** RP
        C <- {f, g}
        C <- pruneComplex C
        N <- coker C#0
-- Module defined by generators and relations:
    if M is a subquotient then
        f <- generators M
        g <- relations M
        f' <- liftUp f
        g' <- liftUp g
        h' <- modulo(f, g)
        e' <- syz h'
        h <- h' ** RP
        e <- e' ** RP
        C <- {h, e}
        C <- pruneComplex C
        N <- coker C#0
    RETURN N
End
  \end{verbatim}
  The command {\tt modulo} returns a matrix $h$ whose image is the pre-image of the image of $g$ under $f$, i.e., 
  $\im(h)=f^{-1}(\im(g))$.
  \end{procedure}

\begin{remark}
  In the implemented package, the isomorphism $N\xto{\sim}M$ is stored in {\tt N.cache.pruningMap}. This is done by
  keeping track of every coordinate change in {\tt pruneUnit}.
  \end{remark}


\subsection{Computing Length and the Hilbert-Samuel Function}
 Recall the definition of the Hilbert-Samuel function from Definition \ref{def:HS}. The first step in computing
 this function is computing length of modules over local rings.

\begin{procedure}[{\tt length}]
 Given a module $M$ over local ring $R_\p$ returns the length of $M$.
\begin{verbatim}
Input:  module M
Output: integer n
Begin
    (RP, m) <- ring of M and its maximal ideal
    s <- 0
    REPEAT
        N <- minimalPresentation M
        n <- number of generators of N
        M <- m * M
        s <- s + n
    UNTIL n = 0
    RETURN s
End
  \end{verbatim}
  \end{procedure}

 Now we can compute the Hilbert-Samuel function of a finitely generated module $M$ with respect to a parameter ideal
 $\q$ as follows:

\begin{procedure}[{\tt hilbertSamuelFunction}]
 Given a parameter ideal $\q$ for module $M$ over local ring $R_\p$ and an integer $n$, returns the value of the
 Hilbert-Samuel function at $n$. Recall that when the parameter ideal is the maximal ideal of $R_\p$, we have 
 $\length(\m^i N/\m^{i+1} N)=\length(\m^i N)$, which is easier to compute.
\begin{verbatim}
Input:  parameter ideal q, module M, integer n
Output: integer r
Begin
    (RP, m) <- ring of M and its maximal ideal
    N <- q^n * M
    if q == m then
        RETURN length N
    else
        RETURN length(N/(q * N))
End
  \end{verbatim}
  \end{procedure}

\begin{example} Consider the twisted cubic curve and an embedded curve defined by ideal $I$:
\begin{verbatim}
i2 : R = ZZ/32003[x,y,z,w];
i3 : P = ideal "  yw-z2,   xw-yz,  xz-y2"
               2                        2
o3 = ideal (- z  + y*w, - y*z + x*w, - y  + x*z)
i4 : I = ideal "z(yw-z2)-w(xw-yz), xz-y2"
               3               2     2
o4 = ideal (- z  + 2y*z*w - x*w , - y  + x*z)
  \end{verbatim}
 One can check that the radical of $I$ is $\p$, hence $I$ is a $\p$-primary ideal.
\begin{verbatim}
i5 : codim I == codim P
o5 = true
i6 : radical I == P
o6 = true
  \end{verbatim}
 In particular, $\p$ is the minimal prime above $I$, hence by Corollary \ref{cor:mengyuan}, $R_\p/IR_\p$ is Artinian.
 Finally, we compute the length and Hilbert-Samuel function of $R_\p/IR_\p$:
\begin{verbatim}
i7 : RP = localRing(R, P);
i8 : N = RP^1/promote(I, RP)
o8 = cokernel | -z3+2yzw-xw2 -y2+xz |
                              1
o8 : RP-module, quotient of RP
i9 : length(N)
o9 = 2
i10 : for i from 0 to 3 list hilbertSamuelFunction(N, i)
o10 = {1, 1, 0, 0}
  \end{verbatim}
  \end{example}






\section{Examples and Applications in Intersection Theory}
 An important application of computation of length over local rings is in intersection theory. In the following
 examples we compute various functions using the {\tt LocalRings} package.


\begin{example} Computing the geometric multiplicity of intersections.

 B\'{e}zout's theorem tells us that the number of points where two curves meet is at most the product of their 
 degrees. Consider the parabola $y=x^2$ and lines $y=x$ and $y=0$:
\begin{verbatim}
i2 : R = ZZ/32003[x,y];
i3 : C = ideal"y-x2"; -- parabola y=x^2
i4 : D = ideal"y-x";  -- line y=x
i5 : E = ideal"y";    -- line y=0
  \end{verbatim}
 The (naive) geometric multiplicity of the intersection of the curve $C$ with the curves $D$ and $E$ at the point 
 $(1,1)$ is given by the length of the Artinian ring $R_{(x-1,y-1)}/(C+D) R_{(x-1,y-1)}$ \cite[pp. 9]{fulton}:
\begin{verbatim}
i6 : use R;
i7 : P = ideal"y-1,x-1";
i8 : RP = localRing(R, P);
i9 : length (RP^1/promote(C+D, RP))
o9 = 1
i10 : length (RP^1/promote(C+E, RP))
o10 = 0
  \end{verbatim}
 Similarly, we can find the geometric multiplicity of intersections at the origin:
\begin{verbatim}
i11 : use R;
i12 : P = ideal"x,y";  -- origin
i13 : RP = localRing(R, P);
i14 : length(RP^1/promote(C+D, RP))
o14 = 1
i15 : length(RP^1/promote(C+E, RP))
o15 = 2
  \end{verbatim}
 Now consider the curves $y=x^2$ and $y=x^3$:
\begin{verbatim}
i2 : R = ZZ/32003[x,y];
i3 : C = ideal"y-x3";
i4 : D = ideal"y-x2";
i5 : E = ideal"y";
  \end{verbatim}
 Once again, we can compute the geometric multiplicity of intersections at the origin:
\begin{verbatim}
i6 : use R;
i7 : P = ideal"x,y";
i8 : RP = localRing(R, P);
i9 : length(RP^1/promote(C+D, RP))
o9 = 2
i10 : length(RP^1/promote(C+E, RP))
o10 = 3
  \end{verbatim}
 And at the point $(1,1)$:
\begin{verbatim}
i11 : use R;
i12 : P = ideal"x-1,y-1";
i13 : RP = localRing(R, P);
i14 : length(RP^1/promote(C+D, RP))
o14 = 1
i15 : length(RP^1/promote(C+E, RP))
o15 = 0
  \end{verbatim}
  \end{example}

\begin{example} Computing the Hilbert-Samuel series.

  Consider the ring $\Z/32003\Z[x,y]_{(x,y)}$ as a module over itself:
\begin{verbatim}
i2 : R = ZZ/32003[x,y];
i3 : RP = localRing(R, ideal gens R);
i4 : N = RP^1;
i5 : q = ideal"x2,y3"
             2   3
o5 = ideal (x , y )
  \end{verbatim}
  First we compute the Hilbert-Samuel series with respect to the maximal ideal:
\begin{verbatim}
i6 : for i from 0 to 5 list hilbertSamuelFunction(N, i) -- n+1
o6 = {1, 2, 3, 4, 5, 6}
  \end{verbatim}
  In particular, $R_\p$ is not of finite length.
  Moreover, we compute the Hilbert-Samuel series with respect to the parameter ideal $(x^2,y^3)$:
\begin{verbatim}
i7 : for i from 0 to 5 list hilbertSamuelFunction(q, N, i) -- 6(n+1)
o7 = {6, 12, 18, 24, 30, 36}
  \end{verbatim}
  \end{example}

\begin{example} Computing multiplicity of a zero dimensional variety at a the origin \cite[Teaching the Geometry of 
  Schemes, \S5]{M2}.

  Consider the variety defined by the ideal $(x^5+y^3+z^3,x^3+y^5+z^3,x^3+y^3+z^5)$:
\begin{verbatim}
i2 : R = QQ[x,y,z];
i3 : RP = localRing(R, ideal gens R);
i4 : I = ideal"x5+y3+z3,x3+y5+z3,x3+y3+z5"
             5    3    3   5    3    3   5    3    3
o4 = ideal (x  + y  + z , y  + x  + z , z  + x  + y )
  \end{verbatim}
  Now we compute the length:
\begin{verbatim}
i5 : M = RP^1/I
o5 = cokernel | x5+y3+z3 y5+x3+z3 z5+x3+y3 |
                              1
o5 : RP-module, quotient of RP
i6 : length(RP^1/I)
o6 = 27
  \end{verbatim}
  Note that the sum of the terms in the Hilbert-Samuel function gives the same number:
\begin{verbatim}
i7 : for i from 0 to 6 list hilbertSamuelFunction(M, i)
o7 = {1, 3, 6, 7, 6, 3, 1}
  \end{verbatim}
  \end{example}





\section{Other Examples from Literature}
 Various other examples of computations involving local rings (as well as updated versions of this thesis) are 
 available online on author's webpage at: \url{https://ocf.berkeley.edu/~mahrud/thesis/examples.m2}

\section{Conclusion and Future Directions}
 Over the course of this thesis, a framework for performing various computations over localization of polynomials
 with respect to a prime ideal has been established. Therefore the next logical steps are extending the framework
 to support localization with respect to arbitrary multiplicative closed sets and improving the efficiency of the
 algorithms. This can be accomplished by developing new theoretical techniques, such as a local monomial order, or 
 by transferring the computations to the \MM engine.

 The author is in particular interested in applying the machinery developed here to study resolutions of 
 singularities and conjectures regarding minimal free resolutions.
 
\section*{Acknowledgment}
 The author thanks thesis advisor David Eisenbud and Mike Stillman, without whom partaking in this honors thesis 
 would not have been possible. The author also thanks Justin Chen, Chris Eur, and Mengyuan Zhang for many helpful 
 comments and discussions in the past year.

\end{document}